\documentclass[12pt,a4paper]{article}

\usepackage{amsfonts, amsmath, amssymb, amsthm}

\usepackage{enumerate,enumitem}

\usepackage{xcolor}

\usepackage{graphics}

\usepackage{array}
\usepackage{etoolbox}

\usepackage{caption,float}

\newcounter{rowcntr}[table]
\renewcommand{\therowcntr}{\arabic{rowcntr}}

\newcolumntype{N}{>{\refstepcounter{rowcntr}\therowcntr}c}

\AtBeginEnvironment{tabular}{\setcounter{rowcntr}{0}}

\newcommand{\N}{{\mathbb{N}}}
\newcommand{\R}{{\mathbb{R}}}
\newcommand{\C}{{\mathbb{C}}}

\newcommand{\xx}{{\mathcal{X}}}
\newcommand{\hh}{{\mathcal{H}}}

\newcommand{\jay}{{j}}

\newcommand{\al}{{\alpha}}
\newcommand{\be}{{\beta}}
\newcommand{\ka}{{\kappa}}
\newcommand{\la}{{\lambda}}

\newcommand{\hrho}{{\widehat{\rho}}}
\newcommand{\hT}{{\widehat{T}}}
\newcommand{\hpsi}{{\widehat{\psi}}}

\newcommand{\hhh}{{\widehat{\hh}}}
\newcommand{\hK}{{\widehat{K}}}
\newcommand{\hla}{{\widehat{\la}}}
\newcommand{\hu}{{\widehat{u}}}
\newcommand{\hv}{{\widehat{v}}}
\newcommand{\hS}{{\widehat{S}}}

\newcommand{\Vu}{{\mathbf{u}}}
\newcommand{\VM}{{\mathbf{M}}}
\newcommand{\VK}{{\mathbf{K}}}

\DeclareMathOperator{\supp}{supp}
\DeclareMathOperator{\spn}{span}

\newcommand{\HS}{{\textnormal{HS}}}

\newcommand{\ra}[1]{\renewcommand{\arraystretch}{#1}}

\theoremstyle{definition}
\newtheorem{dfn}{Definition}

\theoremstyle{remark}

\theoremstyle{plain}
\newtheorem{lem}[dfn]{Lemma}
\newtheorem{prop}[dfn]{Proposition}
\newtheorem*{prop*}{Proposition}
\newtheorem{thm}[dfn]{Theorem}
\newtheorem*{thm*}{Theorem}

\begin{document}

\title{Monte Carlo wavelets: a randomized approach to frame discretization}

\author{
 Zeljko Kereta
 \and
  Stefano Vigogna
  \and
  Valeriya Naumova
  \and
  Lorenzo Rosasco
  \and
Ernesto De Vito
}

\date{}

\maketitle

\begin{abstract}
 In this paper we propose  and study a family of continuous wavelets  on general domains, and a corresponding stochastic discretization that we call  Monte Carlo wavelets. 
 First, using tools from the theory of reproducing kernel Hilbert spaces and associated integral operators,  we define a family of continuous wavelets by spectral calculus. 
 Then, we propose a  stochastic discretization based on  Monte Carlo  estimates of integral operators. Using  concentration of measure results,  we establish the convergence of  such a discretization and  derive convergence rates under natural regularity assumptions.
\end{abstract}

\section{Introduction}

To construct a discrete frame it is common practice to first construct a frame on a continuous parameter (measure) space, where the mathematical properties are nice and the structure is rich, and then obtain a discrete frame by carefully selecting a discrete subset of parameters.
The discretization problem has been studied extensively and in several settings \cite{Fornasier2005,Coulhon2012,FG1,FREEMAN2019784}.
Although theoretically relevant in establishing density theorems and complete characterizations, some general discretization procedures require assumptions which might be difficult to confirm, and they are often not constructive.
On the other hand, discrete parameter selection in constructive discretization methods is usually sensitive,
and it does not generalize trivially from the Euclidean space to more general geometries \cite{CR1}.
In the following, we propose a different approach based on random sampling. 

First, we consider a general domain with an associated positive definite kernel, and exploit 
the theory of reproducing kernel Hilbert spaces (RKHS) and corresponding integral operators  to  define a family of continuous wavelets. Here, we borrow ideas from \cite{blanch,grib,Coulhon2012}, and develop them using the properties of an RKHS.  Then, we propose a stochastic discretization approach which replaces constructive discrete parameter selection with random sampling. We note that the corresponding  discrete frame is actually only finite-dimensional. In particular, it does not solve the discretization problem in the classical sense.  Indeed, each discrete frame is a frame only up to a certain sampling resolution. The frame dimensions increase with the number of samples,  and our main 
technical contribution is proving convergence to the continuous frame as the number of samples goes to infinity. Further, we derive convergence rates under natural regularity assumptions \cite{Feichtinger2016}.
Our analysis combines classical tools from approximation theory and spectral calculus with results of concentration of measure to deal with random sampling. 

The paper is organized as follows.
In Section \ref{sec:setting} we introduce the general framework in which we will conduct our analysis.
This will allow us to state the continuous and discrete settings within a unified formalism.
The key ingredient is a reproducing kernel and the associated integral operator.
Section \ref{sec:frame} contains the general frame construction based on spectral filters and eigenfunctions of the integral operator.
In Section \ref{sec:spectral} we consider a specific class of filters, stemming from the theory of inverse problems, for which we will derive explicit approximation rates.
In Section \ref{sec:consistency} we specialize the construction of Section \ref{sec:frame} in a continuous and a discrete frame,
regarding the latter as a Monte Carlo estimate of the former.
Our main result, Theorem \ref{thm:main}, establishes quantitative consistency of the Monte Carlo estimate on a class of Sobolev-regular signals.
In Section \ref{sec:numeric} we provide an implementable formula to compute our Monte Carlo wavelets from the eigendecomposition of a sample kernel matrix.

\begin{table}[H]
\caption*{Notation} \vspace{-10pt}
\label{tab:notation}
\centering
\small{
\ra{1.5}
\resizebox{0.66\columnwidth}{!}{
\begin{tabular}{l l}

symbol & definition \\

\hline

$\langle \cdot , \cdot \rangle$ & inner product in a Hilbert space \\

$\|\cdot\|$ & norm of a Hilbert space, or the operator norm \\

$\|\cdot\|_\HS$ & Hilbert-Schmidt norm \\

$\spn\{S\}$ & linear span of the set $S$ \\

$\overline{S}$ & topological closure of the set $S$ \\

$P_S$ & orthogonal projection onto the closed subspace $S$ \\

$\supp(\rho)$ & support of the measure $\rho$ \\

$\delta_x$ & Dirac measure at $x$ \\

$X \lesssim Y$ & $ X \le C Y $ for some constant $C>0$ \\

$ \Vu[i] $ & $i$-th component of the vector $ \Vu \in \R^N $ \\

$ \VM[{i,j}] $ & $(i,j)$-th component of the matrix $ \VM \in \R^{N\times N} $ \\

\hline

\end{tabular}
}
}
\end{table}

\section{Setting} \label{sec:setting}

Let $\xx$ be a locally compact, second countable topological space endowed with a Borel probability measure $\rho$.
We fix a continuous positive semi-definite kernel
\begin{equation} \label{eq:K}
K : \xx \times \xx \to \C ,
\end{equation}
and let $\hh$ be the corresponding reproducing kernel Hilbert space
\begin{equation} \label{eq:H}
 \hh = \{ f : \xx \to \C : f(x) = \langle f , K_x \rangle \} = \overline{\spn} \{ K_x \}_{x\in\xx} ,
\end{equation}
with $ K_x = K(\cdot,x) $ and inner product given by $ \langle K_x , K_y \rangle = K(x,y) $.
Assuming $ \sup_{x\in X} K(x,x) \le \ka^2 $,
we define the integral operator $ T : \hh \to \hh $
\begin{equation} \label{eq:T}
 Tf(x) = \int_\xx K(x,y) f(y) d\rho(y) .
\end{equation}
Under these assumptions, $T$ is a positive self-adjoint trace class operator with spectrum $ \sigma(T) \subset [0,\ka^2] $.
Let
\begin{equation*}
 T v_i = \begin{cases}
              \la_i v_i & i\in I_\rho \\
              \hspace{1pt} 0 \ v_i & i\in I_0
             \end{cases}
\end{equation*}
be the spectral decomposition of $T$ over the orthonormal basis $\{v_i\}_{i\in I_\rho\cup I_0}$.
The subspace
\begin{equation} \label{eq:Hrho}
 \hh_\rho = \overline{\spn} \{ K_x \}_{x \in \supp(\rho)} = \overline{\spn} \{ v_i \}_{i\in I_\rho}
\end{equation}
splits $\hh$ as $ \hh = \hh_\rho \oplus \ker T $.

As examples of this setting, we can think of $\xx$ as $\R^d$, or non-Euclidean domains such as a compact connected Riemannian manifold or a weighted graph.
In all these cases, we can take $K$ as the heat kernel associated with the proper notion of Laplacian, be it the Laplace-Beltrami operator or the graph Laplacian.

\section{Frame construction} \label{sec:frame}
 
To construct our frame, we take inspiration from \cite{blanch} and references therein.
 We start by defining filters on the spectrum of $T$.
 Let $ (G_\jay)_{\jay\in\N} $ be a family of bounded measurable functions $ G_\jay : [0,\infty) \to [0,\infty) $
 satisfying
 \begin{align} \label{eq:G}
  \sum_{\jay\in\N} \la G_\jay(\la)^2 = 1 \qquad \la \in (0,\ka^2] .
 \end{align}
$G_\jay(T)$ is then a positive, self-adjoint operator on $\hh$
with $ \sigma(G_\jay(T)) = G_\jay(\sigma(T)) \\ \subset [ 0 , G_\jay(\ka^2) ] $, for every $\jay$.
We can thus define
\begin{equation} \label{eq:frame}
 \psi_{\jay,x} = G_\jay(T) K_x \qquad \jay \in \N , x \in \xx .
\end{equation}
Note that, if $ x \in \supp(\rho) $, then $ \psi_{\jay,x} \in \hh_\rho $ since $ G_\jay(T) \hh_\rho \subset \hh_\rho $.
Moreover, using self-adjointness of $G_\jay(T)$, the reproducing property \eqref{eq:H}, and expression \eqref{eq:Hrho},  we can rewrite $ \psi_{\jay,x} $ as
\begin{equation} \label{eq:wavelet}
 \psi_{\jay,x} = \sum_{i\in I_\rho} G_\jay(\la_i) \overline{v_i(x)} v_i .
\end{equation}
This reveals \eqref{eq:frame} as a generalization of classical wavelets,
where $v_i$'s correspond to eigenfunctions of the Laplacian.
We may therefore interpret $x$ and $\jay$ as location and scale parameters, respectively.

The following proposition shows that the family defined through \eqref{eq:frame} is a Parseval frame on $\hh_\rho$.
 \begin{prop} \label{prop:frame}
 For every $ f \in \hh $,
  $$
  \| P_{\hh_\rho} f \|^2 = \sum_{\jay\in\N} \int_\xx | \langle f , \psi_{\jay,x} \rangle |^2 d\rho(x) .
  $$
 \end{prop}
 \begin{proof}
 Let $ f \in \spn \{v_i\}_{i\in I_\rho\cup I_0} $.
 Since $G_\jay(T)$ is self-adjoint, we have $ \langle f , \psi_{\jay,x} \rangle = \langle G_\jay(T) f , K_x \rangle $.
 Hence, integrating over $\xx$ we get
 \begin{align*}
  \int_\xx | \langle f , \psi_{\jay,x} \rangle |^2 d\rho(x) & = \langle T G_\jay(T) f , G_\jay(T) f \rangle \\
  & =\langle T G_\jay(T)^2 f , f \rangle \\
  & = \sum_{i\in I_\rho} \la_i G_\jay(\la_i)^2 | \langle f , v_i \rangle |^2 .
 \end{align*}
 Summation over $\N$ along with \eqref{eq:G} thus gives
 \begin{align*}
  \sum_{\jay\in\N}  \langle G_\jay(T)^2 T f , f \rangle = \sum_{i\in I_\rho} | \langle f , v_i \rangle |^2 \sum_{\jay\in\N} \la_i G_\jay(\la_i)^2
  = \sum_{i\in I_\rho} | \langle f , v_i \rangle |^2 = \| P_{\hh_\rho} f \|^2 .
 \end{align*}
The assertion follows since $ \spn \{v_i\}_{i\in I_\rho\cup I_0} $ is dense in $\hh$.
 \end{proof}

\section{Examples of spectral filters} \label{sec:spectral}

Let $ (g_\jay)_{\jay\in\N} $ be a family of functions $ g_\jay : [0,\infty) \to [0,\infty) $ such that
\begin{equation} \label{eq:g}
 0 \le g_\jay \le g_{\jay+1} , \qquad \lim_{\jay\to\infty} \la g_\jay(\la) = 1 .
\end{equation}
We can define filters obeying \eqref{eq:G} by means of \eqref{eq:g}, as
\begin{equation*}
 G_0(\la) = \sqrt{g_0(\la)} , \qquad G_{\jay+1} (\la) = \sqrt{ g_{\jay+1}(\la) - g_{\jay}(\la) } .
\end{equation*}

The \emph{qualification} of a spectral function $g_\jay(\la)$ is a constant $ \overline{\nu} \in (0,\infty] $ such that, for all $\nu>0$,
\begin{equation*}
 \sup_{\la\in(0,\ka^2]} \la^\nu (1 - \la g_\jay(\la)) \lesssim \jay^{-\min\{\nu,\overline{\nu}\}} .
\end{equation*}
As instances of $g_\jay(\la)$, we can take spectral functions in Table \ref{tab:g} (see also \cite{MR}).

\begin{table}[H]
\caption{} \vspace{-10pt}
\label{tab:g}
\centering
\small{
\ra{1.5}

\begin{tabular}{l c c}

\multicolumn{1}{c}{method} & $g_\jay(\lambda)$ & qualification \\

\hline

Tikhonov regularization & $\frac{1}{\la + 1/\jay} $ & $1$ \\

Iterated Tikhonov & $ \frac{ (\la + 1/\jay)^m - (1/\jay)^m }{\la(\la + 1/\jay)^m } $ & $m$ \\

Landweber iteration & $ \frac{1}{\la} ( 1 - (1 - \gamma \la)^{\jay} ) $ & $\infty$ \\

Asymptotic regularization & $ \frac{1}{\la} ( 1 - \exp(- \jay \la) ) $ & $\infty$ \\

\hline

\end{tabular}
}
\end{table}

\noindent
These functions play a fundamental role in ill-posed inverse problems regularization,
where the error of a given regularizer decays with a rate governed by its qualification.

\section{Consistency} \label{sec:consistency}

From now on, assume $ \supp{\rho} = \xx $,
whence $ \hh_\rho = \hh $ by \eqref{eq:Hrho}.
Therefore, \eqref{eq:frame} is a frame on the whole space $\hh$, thanks to Proposition \ref{prop:frame}.
Suppose we are given $N$ independent and identically distributed samples
\begin{equation} \label{eq:samples}
 x_1,\dots,x_N \sim \rho .
 \end{equation}
Then we can consider the empirical distribution
\begin{equation*}
\hrho_N = \frac{1}{N} \sum_{i=1}^N \delta_{x_i} ,
\end{equation*}
define the empirical integral operator
\begin{equation} \label{eq:That}
 \hT f (x) = \frac{1}{N} \sum_{i=1}^N K(x,x_i) f(x_i) ,
\end{equation}
and repeat the construction of Section \ref{sec:frame}, to get
\begin{equation} \label{eq:framehat}
 \hpsi_{\jay,x_i} = G_\jay(\hT) K_{x_i} \qquad \jay \in \N , i = 1,\dots,N .
\end{equation}
In view of Proposition \ref{prop:frame}, \eqref{eq:framehat} defines a discrete Parseval frame on
\begin{equation*}
 \hhh_N = \hh_{\hrho_N} = \spn\{ K_{x_i} \}_{i=1}^N \simeq \R^N .
\end{equation*}
Since \eqref{eq:That} is a Monte Carlo estimate of \eqref{eq:T},
we call \eqref{eq:framehat} a family of \emph{Monte Carlo wavelets}.
Note that, as we take more and more samples, we obtain a sequence of discrete Parseval frames
on a sequence of nested subspaces of increasing dimension:
\begin{equation*}
  \hhh_N \subset \hhh_{N+1} \subset \dots \subset \hh .
\end{equation*}
Our goal is to establish consistency of the Monte Carlo estimate \eqref{eq:framehat},
that is, to see whether, and in what sense, \eqref{eq:framehat} tends to \eqref{eq:frame} as $ N \to \infty $.

Now, let
\begin{equation} \label{eq:Tom}
 T_\jay f = \int_\xx \langle f , \psi_{\jay,x} \rangle \psi_{\jay,x} d\rho(x), \quad \hT_\jay f = \frac{1}{N} \sum_{i=1}^N \langle f , \hpsi_{\jay,x_i} \rangle \hpsi_{\jay,x_i}
\end{equation}
be the frame operator at scale $\jay$ and its empirical counterpart.
Proposition \ref{prop:frame} gives a resolution of the identity,
which can be approximated by a truncated empirical version:
\begin{equation} \label{eq:id}
 I_\hh = \sum_{\jay\in\N} T_\jay \approx \sum_{\jay\le\tau} \hT_\jay .
\end{equation}
We split the resolution error of \eqref{eq:id} on a signal $f\in\hh$ into the approximation and the sampling error,
\begin{equation} \label{eq:split}
 \| f - \sum_{\jay\le\tau} \hT_\jay f \| \le \| \sum_{\jay>\tau} T_\jay f \| + \| \sum_{\jay\le\tau} (T_\jay - \hT_\jay ) f \| ,
\end{equation}
and derive quantitative bounds for each term (Propositions \ref{prop:approximation} and \ref{prop:sampling}).
Tuning the resolution $\tau$ in terms of the sample size $N$
will then yield our consistency result in Theorem \ref{thm:main}.

For what concerns the approximation error, note that
Proposition \ref{prop:frame} already implies
\begin{equation*}
\| \sum_{\jay>\tau} T_\jay f \| \xrightarrow{\tau\to\infty}0
\end{equation*}
by dominated convergence.
However, in order to obtain a convergence rate,
we need to assume some regularity of the signal $f$.
Specifically, we will assume the following Sobolev smoothness condition \cite{Feichtinger2016}:
\begin{equation*}
  f = T^\al h \qquad \text{for some } h \in \hh, \al > 0 .
 \end{equation*}

\begin{prop}[Approximation error] \label{prop:approximation}
Let $ g_\jay(\la) $ as in \eqref{eq:g}.
Suppose $ g_j(\la) $ has qualification $ \overline{\nu} $.
 Then, for every $ f \in T^\al\hh $,
 $$
  \| \sum_{\jay>\tau} T_\jay f \| \lesssim \|T^{-\al}f\| \ka^{2(\al-\be)} \tau^{-\be} ,
  $$
 where $ \be = \min\{\al,\overline{\nu}\} $.
\end{prop}
\begin{proof}
 By Lemma \ref{lem:Tom} we have $ \sum_{\jay\le\tau} T_\jay = T g_\tau(T) $,
 hence
\begin{align*}
  \| \sum_{\jay>\tau} T_\jay f \|^2 & = \sum_i | 1 - \la_i g_\tau(\la_i) |^2 | \langle f , v_i \rangle |^2 \\
  & = \sum_i \la_i^{2\be} | 1 - \la_i g_\tau(\la_i) |^2 | \langle T^{-\be}f , v_i \rangle |^2 .
\end{align*}
Thus,
\begin{align*}
\| \sum_{\jay>\tau} T_\jay f \| & \le \| T^{-\be} f \| \sup_i \la_i^\be (1 - \la_i g_\tau(\la_i))
\lesssim  \| T^{-\be} f \| \tau^{-\be} . \qedhere
\end{align*}
\end{proof}

Bound on the sampling error follows by concentration of the empirical integral operator $\widehat T$ on the continuous operator  $T$.
\begin{prop}[Sampling error] \label{prop:sampling}
Let $ g_j(\la) $ be as in \eqref{eq:g}.
Suppose $ \la \mapsto \la g_j(\la) $ is Lipschitz continuous on $ [0,\ka^2] $ with Lipschitz constant bounded by $j$.
Then, for every $ f \in \hh $,  with probability higher than $ 1 - 2 e^{-t} $,
 $$
  \| \sum_{\jay\le\tau} (T_\jay - \hT_\jay ) f \| \lesssim \|f\| \ka^2 \sqrt{t} \tau N^{-1/2} .
 $$
\end{prop}
\begin{proof}
 Using Lemma \ref{lem:Lip} we have
 \begin{align*}
  \| \sum_{\jay\le\tau} (T_\jay - \hT_\jay ) f \| & \le \| T g_\tau(T) - \hT g_\tau(\hT) \| \|f\| \\
  & \le \| T g_\tau(T) - \hT g_\tau(\hT) \|_\HS \|f\| \\
  & \le \tau \| T - \hT \|_\HS \|f\| .
 \end{align*}
Thanks to the reproducing property \eqref{eq:H}, we can write
$$
 T = \int_\xx K_x \otimes K_x d\rho(x) , \qquad \hT = \frac{1}{N} \sum_{i=1}^N K_{x_i} \otimes K_{x_i} .
$$
This allows for the concentration in Lemma \ref{lem:T-That}, which yields
 $$
 \| T - \hT \|_\HS \lesssim \ka^2 \sqrt{t} N^{-1/2}
 $$
 with probability higher than $ 1 - 2 e^{-t} $ .
\end{proof}

We are finally ready to state the consistency of our Monte Carlo discretization.

\begin{thm} \label{thm:main}
 Let $ g_\jay(\la) $ as in \eqref{eq:g}.
 Suppose $ g_j(\la) $ has qualification $ \overline{\nu} $,
 and $ \la \mapsto \la g_j(\la) $ is Lipschitz continuous on $ [0,\ka^2] $ with Lipschitz constant bounded by $j$.
 Then, for every $ f \in T^\al\hh $,
   with probability higher than $ 1 - 2 e^{-t} $,
 \begin{align*}
 \| f - \hspace{-10pt} \sum_{\jay\le N^{\frac{1}{2\be+2}}} \sum_{i=1}^N &\langle f , \hpsi_{\jay,x_i} \rangle \hpsi_{\jay,x_i} \|
  \lesssim  \|T^{-\al}f\| ( \ka^{2(\al-\be)} + \ka^{2\al+2} \sqrt{t} ) N^{-\frac{\be}{2\be+2}} ,
 \end{align*}
where $ \be = \min\{\al,\overline{\nu}\} $.
\end{thm}
\begin{proof}
 We split $ \| f - \sum_{\jay\le\tau} \hT_\jay f \| $ in \eqref{eq:split},
 and combine the bounds in Propositions \ref{prop:approximation} and \ref{prop:sampling}.
In order to balance the two terms, we choose $ \tau = N^{\frac{1}{2\be+2}} $, which concludes the proof.
\end{proof}

In particular, in view of Lemma \ref{lem:lip}, we can use the spectral filters defined in Table \ref{tab:g},
for which we obtain the following rates:

\begin{table}[th]
\caption{} \vspace{-10pt}
\label{tab:rates}
\centering
\small{
\ra{1.5}

\begin{tabular}{l l}

\multicolumn{1}{c}{method} & \multicolumn{1}{c}{rate} \\

\hline

Tikhonov regularization & $ N^{-\frac{\min\{\al,1\}}{2\min\{\al,1\}+2}} $ \\

Iterated Tikhonov & $ N^{-\frac{\min\{\al,m\}}{2\min\{\al,m\}+2}}  $ \\

Landweber iteration & $ N^{-\frac{\al}{2\al+2}} $ \\

Asymptotic regularization & $ N^{-\frac{\al}{2\al+2}}$ \\

\hline

\end{tabular}
}
\end{table}

\section{Numerical implementation} \label{sec:numeric}

To implement the system of Monte Carlo wavelets in \eqref{eq:framehat}
we exploit equation \eqref{eq:wavelet}, that is,
 \begin{equation*} \label{eq:wavelethat}
 \hpsi_{j,x_k} = \frac{1}{N} \sum_{i=1}^N G_j(\hla_i) \hv_i (x_k) \hv_i ,
\end{equation*}
where $ \hT \hv_i = \hla_i \hv_i $ is the spectral decomposition of $\hT$.
In particular, we need to compute eigenvalues and eigenvectors of $\hT$,
translating the calculations from $\hhh$ to $\R^N$ and back \cite{RBD10}.
To this end, we introduce the restriction map $ \hS : \hhh \to \R^N $
\begin{equation}
 (\hS f)[i] = f(x_i) \qquad i = 1,\dots,N ,
\end{equation}
whose adjoint $ \hS^* : \R^N \to \hhh $ defines the out-of-sample extension
\begin{equation*}
 \hS^* \Vu = \frac{1}{N} \sum_{\ell=1}^N \Vu[\ell] K_{x_\ell} .
\end{equation*}
Let
\begin{equation} \label{eq:Khat}
\widehat\VK[i,k] = \frac{1}{N} K(x_i,x_k), \qquad i,k=1,\dots,N
\end{equation}
be the (normalized) sample kernel matrix.
Now $ \hT = \hS^* \hS $ and $ \widehat\VK = \hS \hS^* $,
hence $\hT$ and $\widehat\VK$ share the same eigenvalues.
Moreover, the eigenvectors of $\hT$ are related to the ones of $\widehat\VK$
by the singular value decomposition of $\hS$:
\begin{equation*}
 \hS \hv_i = \sqrt{\hla_i} \hu_i, \qquad i = 1,\dots,N .
\end{equation*}
Thus, we have
\begin{equation*}
 \hv_i = \hla_i^{-\frac{1}{2}} \hS^* \hu_i = \hla_i^{-\frac{1}{2}} \frac{1}{N} \sum_{\ell=1}^N \hu_i[\ell] K_{x_\ell} ,
\end{equation*}
which evaluated at $ x_k $ gives
\begin{align*}
 \hv_i(x_k) = \hla_i^{-\frac{1}{2}} \frac{1}{N} \sum_{\ell=1}^N K(x_k,x_\ell) \hu_i[\ell]
  = \hla_i^{-\frac{1}{2}} (\hK \hu_i)[k] \! = \! \sqrt{\hla_i} \hu_i[k] .
\end{align*}
We therefore obtain
\begin{equation} \label{eq:wavelethat}
 \hpsi_{j,x_k} (x) = \frac{1}{N} \sum_{i,\ell=1}^N G_j(\hla_i) \hu_i[k] \hu_i[\ell] K(x,x_\ell) .
\end{equation}
In summary,
given a kernel \eqref{eq:K}, filters $\eqref{eq:G}$ and samples \eqref{eq:samples},
a numerical implementation of the Monte Carlo wavelets \eqref{eq:framehat}
can be performed from the sample kernel matrix \eqref{eq:Khat} alone
following the steps below:
\begin{enumerate}[label=\arabic*.]
\item compute eigenvalues and eigenvectors $(\hla_i,\hu_i)_{i=1}^N$ of $\hK$;
\item apply equation \eqref{eq:wavelethat}.
\end{enumerate}

%\section{Conclusions}
%The conclusions go here.

\newpage

\section*{Lemmata}

\begin{lem} \label{lem:Tom}
Let $T$, $G_\jay$ and $T_\jay$ be defined as in \eqref{eq:T}, \eqref{eq:G} and \eqref{eq:Tom}.
Then $ T_\jay = T G_\jay(T)^2 $.
\end{lem}
\begin{proof}
For every $i\in I$ we have
 \begin{align*}
  T_j v_i & = \int_\xx \langle v_i , G_j(T)K_x \rangle G_j(T)K_x d\rho(x) = G_j(\la_i) \int_\xx \langle v_i , K_x \rangle G_j(T)K_x d\rho(x) .
 \end{align*}
 Thus,
 \begin{align*}
  \langle T_j v_i , v_i \rangle & = G_j(\la_i)^2 \int_\xx \langle v_i , K_x \rangle \langle K_x , v_i \rangle d\rho(x) \\
  & = G_j(\la_i)^2 \langle Tv_i , v_i \rangle \\
  & = \langle T G_j(T)^2 v_i , v_i \rangle .
 \end{align*}
Since $ \overline{\spn}\{v_i\}_{i\in I} = \hh $, the claim follows by polarization.
\end{proof}

\begin{lem} \label{lem:Lip}
 Let  $A,B$ be self-adjoint operators on a Hilbert space $\hh$,
 and let $ F : \R \to \C $ be a Lipschitz continuous function with Lipschitz constant $L$.
 Then
 $$
 \| F(A) - F(B) \|_\HS \le L \| A - B \|_\HS .
 $$
\end{lem}
\begin{proof}
 Let $ \{e_i\}_{i\in I} $ and $ \{f_j\}_{j\in J} $ be orthonormal bases of $\hh$
 such that $ Ae_i = \la_i e_i $ and $ Bf_j = \mu_j f_j $.
 Then
 \begin{align*}
  \| F(A) - F(B) \|_\HS^2
  & = \sum_{i\in I , j\in J} | \langle (F(A) - F(B)) e_i , f_j \rangle |^2 \\
  & = \sum_{i\in I , j\in J} | F(\la_i) - F(\mu_j) |^2 | \langle e_i , f_j \rangle |^2 \\
  & \le L^2 \sum_{i\in I , j\in J} | \la_i - \mu_j |^2 | \langle e_i , f_j \rangle |^2 \\
  & = L^2 \| A - B \|_\HS . \qedhere
 \end{align*}
\end{proof}

\begin{lem}[{\cite[Theorem 7]{RBD10}}] \label{lem:T-That}
 The operators $T$ and $\hT$ defined in \eqref{eq:T} and \eqref{eq:That} are Hilbert-Schmidt and
 $$
  \| T - \hT \|_\HS \le 2\sqrt{2} \ka^2 \sqrt{t} N^{-1/2}
 $$
 with probability higher than $ 1 - 2 e^{-t} $.
\end{lem}

\begin{lem} \label{lem:lip}
 Let $ g_\jay(\la) $ be as in Table \ref{tab:g}.
 The function $ \la \mapsto \la g_\jay(\la) $ is Lipschitz continuous on $[0,\ka^2]$,
with Lipschitz constant bounded by $\jay$.
\end{lem}
\begin{proof}
The claim follows by a direct computation of the derivative of $ \la \mapsto \la g_\jay(\la) $.
\end{proof}

\bibliographystyle{plain}
\bibliography{biblio}

% that's all folks

\end{document}